\documentclass{sn-jnl}
\usepackage{amsmath, amsthm, amscd, amsfonts, amssymb, graphicx, color}
\usepackage{caption}
\usepackage{mathtools}
\usepackage{hyperref}
\usepackage{MnSymbol}
\jyear{2024}%

\theoremstyle{thmstyleone}%
\newtheorem{theorem}{Theorem}[section]% meant for sectionwise numbers
\newtheorem{proposition}[theorem]{Proposition}% 

\theoremstyle{thmstyletwo}%

\theoremstyle{thmstylethree}%
\newtheorem{definition}{Definition}[section]%
\newtheorem{remark}{Remark}[section]%
\newtheorem{lemma}{Lemma}[section]
\newtheorem{corollary}{Corollary}[section]
\numberwithin{equation}{section}
\newtheorem{solution}{Solution}[section]%
\renewenvironment{proof}{{\bfseries Proof}}{\qed}

\begin{document}

\title[Ergodic trading model and call option pricing]{Some Applications of Log-Ergodic Processes: Ergodic Trading Model and Call Option Pricing Using the Irrational Rotation}


\author*[1]{\fnm{Kiarash} \sur{Firouzi}}\email{kiarashfirouzi91@gmail.ac.ir}

\author[2]{ \fnm{Mohammad} \spfx{Jelodari} \sur{Mamaghani}}\email{j\_mamaghani@atu.ac.ir}
%\equalcont{The author critically revised the work.}


\affil*[1]{\orgdiv{Department of Mathematics}, \orgname{Allameh Tabataba'i University}, \orgaddress{\street{Dehkadeh Olympic}, \city{Tehran}, \postcode{1489684511}, \state{Tehran}, \country{Iran}}}

\affil[2]{\orgdiv{Department of Mathematics}, \orgname{Allameh Tabataba'i University}, \orgaddress{\street{Dehkadeh Olympic}, \city{Tehran}, \postcode{1489684511}, \state{Tehran}, \country{Iran}}}

%\affil[3]{\orgdiv{Department}, \orgname{Organization}, \orgaddress{\street{Street}, \city{City}, \postcode{610101}, \state{State}, \country{Country}}}


\abstract{
Due to the increasing popularity of futures trading among financial market participants, the risk management of these instruments is crucial. In this paper, we introduce a model for estimating the ideal time for leaving a trading position on a stock. Also, using ergodic theorems, we investigate the European call option pricing problem using a stochastic irrational rotation on the unit circle. Utilizing the properties of log-ergodic processes, we use the time average of the stochastic process of risky assets instead of expectations in our calculations.}




\maketitle
\section{Introduction}
The use of ergodic processes in financial modeling has evolved significantly over the past decades. Early research focused primarily on the theoretical underpinnings of ergodic systems in economics and their application to macroeconomic models. However, more recent studies have explored their utility in financial markets, particularly in developing trading strategies and pricing financial derivatives.

One of the key developments in this area was introduced by Ornstein and Uhlenbeck \cite{orn} with their work on mean-reverting stochastic processes. Their model laid the groundwork for subsequent research on ergodic processes in finance. Feller \cite{feller} extended this concept by exploring the implications of ergodicity for stochastic differential equations, furthering the application of these models in finance.

The relevance of ergodic processes in option pricing was significantly advanced by Cox, Ingersoll, and Ross \cite{cox} with the development of the Cox-Ingersoll-Ross (CIR) model. Their model, which incorporates mean reversion, provided a foundation for the pricing of interest rate derivatives and has been widely adopted in the literature.

In recent years, several studies have expanded on the classical models by incorporating various modifications to better capture market dynamics. For instance, Carr and Wu \cite{carr} introduced a framework that combines stochastic volatility with ergodic processes, providing a more comprehensive approach to pricing options. Similarly, Fouque et al. \cite{fou} developed a multiscale stochastic volatility model that integrates ergodicity and has been shown to improve pricing accuracy for options on assets with mean-reverting properties.

The futures market is one of the most popular active markets in the world. The beginning of the first real organized futures market dates back to 1710 and to the Dōjima Rice Exchange located in Osaka, Japan \cite{105}. Today, futures market has a nature far beyond their agricultural origins. With the existence of the New York Mercantile Exchange, the trading and hedging of financial products using futures contracts have played a significant role in the global economic system, reaching the highest trading volume of over 5.1 trillion dollars per day in 2005 historically \cite{106}. The highest record of trading volume per day in 2024 for trading and hedging of financial products using futures contracts is held by E-mini S\&P 500 futures on Chicago Mercantile Exchange (CME). The average daily trading volume reached over 400 billion dollars per day, given the average price of each contract of around 200,000 dollars \cite{tv1,tv2}. 
%The New York Stock Exchange partnered with the Amsterdam-Brussels-Lisbon-Paris Stock Exchange (Euronext Electronic Exchange) to form the first transcontinental stocks and options exchange in 2006. The formation of Euronext, plus the growth of Internet futures trading platforms created by developed companies, indicates the increasing trend of competition in online futures and options services in the coming years \cite{108}. In terms of trading volume, the National Stock Exchange of India in Mumbai is the largest futures stock exchange in the world \cite{107}.

%Despite its popularity, the high risk of futures trading is considered a deterrent factor for many risk-averse traders \cite{106,109}. Generally, predicting the behavior of financial markets is difficult, although research using machine learning is available in this field \cite{110}. 

Harrison Hong presented a model for the futures market by analyzing the returns and trading patterns of its participants \cite{111}. Binh Do, \cite{113}, studied some trading models, consisting of long and short positions, on a pair of shares (risky assets). Researchers have studied futures trading in terms of risk and return \cite{gor}. However, very few resources are available for estimating the appropriate time to take a trading position on a stock.
% Rama Cont and his colleagues have described the risk and return of dynamic trading in the form of fluctuations of the market value of the basket from a reference level, and based on the price fluctuation, relative to that reference level, in a path-dependent manner, they have provided a framework for risk analysis of trading \cite{112}.
 In this paper, we use the method of \cite{112} and the concept of log-ergodic process \cite{116} to model futures trading on a risky asset and estimate the ideal time to take a trading position. To this end, we assume that the price process of a risky asset follows a log-ergodic process. Regarding the price process as an ergodic dynamical system on the unit circle (Irrational rotation), we substitute the time average of the price process with mathematical expectation in our calculations. Artur Avila et al have studied the behavior of the random walk process using the irrational rotation and have presented a new proof of the J. Beck central limit theorem \cite{114}.

The rest of the paper is as follows:

We introduce some necessary concepts in section \ref{sec2}. In section \ref{sec3}, we present a trading model. Using the properties of the partially ergodic processes, we introduce the dynamics of recurrence times and a time interval for leaving a trading position. In section \ref{sec4}, we define the irrational rotation on the unit circle and study its properties. In section \ref{sec5}, we study the European call option pricing problem using the irrational rotation. In section \ref{sec6}, we solve the Black-Scholes ergodic partial differential equation presented in our recent work \cite{116}. In section \ref{sec7}, we conclude the paper and suggest some directions for future research.

\section{Preliminaries}\label{sec2}
Throughout the paper, we use the filtered probability space
$(\Omega,\mathcal{F},\mathbb{P},(\mathcal{F}_t)_{t\geq0})$,  in which   $\Omega$ is the space of events, $\mathcal{F}$ is a $\sigma$-algebra,  $\mathbb{P}$ is a probability measure, and sub $\sigma$-algebra
$\mathcal{F}_t$ represents the information of the financial market up to time $t$, which is generated by the standard Wiener process $W_t$.

\begin{definition}(Log-ergodicity)\label{logergodic}
A positive stochastic process $X_t$ is considered logarithmically ergodic (log-ergodic) if its logarithm is mean ergodic. More precisely, the positive stochastic process $X_t$ is log-ergodic if the process $Y_t=\ln(X_t)$ satisfies
	\begin{equation}\label{limer}
		\overline{<Y>}:=\lim_{T\rightarrow \infty}\frac{1}{T}\int_0^T (1-\frac{\tau}{T})\mathbf{Cov}_{yy}(\tau)d\tau=0, \quad \forall\tau\in[0,T].
	\end{equation}
	Where $\mathbf{Cov}_{yy}(\tau)$ is the covariance of  $Y_\tau$.
\end{definition}
For more information on this and related concepts refer to \cite{116}.
\begin{definition}\label{deferc}
	Let $W_t$ be a standard Wiener process and $\beta>\frac{3}{2}$. For all $t,s\in[0,T]$, we define the ergodic maker operator of the process\\ $Y_t^\prime=Y_0^\prime+D_t+R_t$ as
	\begin{equation}\label{xi}
		Z_\delta:=\xi_{\delta,W_\delta}^{\beta}[Y_t^\prime]:=0\cdot Y_0^\prime+\frac{W_T}{T^{\beta}}\cdot D_\delta+\frac{1}{T^{\beta}}\cdot R_\delta=D_\delta^{T,W_T}+R_\delta^T,
	\end{equation}
	where $\delta=t-s$ for $t>s$.
\end{definition}
We introduce the definition of the inverse of the ergodic maker operator, which we use in section \ref{sec6}.
\begin{definition}(Inverse of Ergodic Maker Operator)
	For any mean ergodic random process $Z_\delta=D_\delta^{T,W_T}+R_\delta^{T}$ we define the IEMO as follows
	\begin{equation}\label{IEMO}
		\xi_{t,W_t}^{-\beta}[Z_\delta]=c+ \frac{T^\beta}{W_T}\cdot D_t^{T,W_t}+T^{\beta}\cdot R_t^{T},\quad t>0,
	\end{equation}
	where $c$ is a constant. 
\end{definition}
\begin{lemma}\label{key}
	For all $c>0$ and any mean ergodic process $Y_t=D_t+R_t$,
	$$\xi_{t,W_t}^{-\beta}\big[\xi_{\delta,W_\delta}^\beta[c+Y_t]\big]=c+Y_t,$$
	where $c$ is a constant.
\end{lemma}
\begin{proof}
	We have
	\begin{align*}
		\xi_{\delta,W_\delta}^\beta[c+Y_t]= 0\cdot c+\frac{W_T}{T^{\beta}}D_\delta+\frac{1}{T^{\beta}}R_\delta.
	\end{align*}
	Now, from the definition \ref{IEMO} we have
	\begin{align*}
		\xi_{t,W_t}^{-\beta}[\frac{W_T}{T^{\beta}}D_\delta+\frac{1}{T^{\beta}}R_\delta]&=c+ \frac{T^{\beta}}{W_T}\frac{W_T}{T^{\beta}}D_t+T^\beta\frac{1}{T^{\beta}}R_t=c+D_t+R_t\\
		&=c+Y_t.
	\end{align*}
\end{proof}
\begin{proposition}
	For any finite mean ergodic process $Z_\delta$, the coefficient $c$ defined in \ref{IEMO} is unique.
\end{proposition}
\begin{proof}
	Assume that there exist numbers $c_1, c_2> 0$. From lemma \ref{key} for finite the processes $c_1+Y_t$ and $c_2+Y_t$ we have 
	\begin{equation}\label{a}
	\xi_{\delta,W_\delta}^\beta[c_1+Y_t]=\xi_{\delta,W_\delta}^\beta[c_2+Y_t].
	 \end{equation}
 The equality \ref{a} is true because the EMO drops the constants $c_1$ and $c_2$. Now, using \ref{IEMO} yields
	\begin{align*}
	\xi_{t,W_t}^{-\beta}\big[\xi_{\delta,W_\delta}^\beta[c_1+Y_t]\big]&=	\xi_{t,W_t}^{-\beta}\big[\xi_{\delta,W_\delta}^\beta[c_2+Y_t]\big]\\
	c_1+Y_t&=c_2+Y_t\\
	c_1&=c_2.			
	\end{align*}
\end{proof}
\begin{remark}
	Note that from uniqueness of the coefficient $c$ in the definition \ref{IEMO} implies that IEMO is well-defined.
\end{remark}
\begin{theorem}(Kac)\label{kac}
	Let $f$ be a measure preserving transformation, and $A\in\mathcal{F}$ such that $\mathbb{P}(A)>0$. Define $\rho_A(\omega)=\min\{n\in \mathbb{N}\lvert f^n(\omega)\in A\}$. Then, 
	$$\mathbb{E}[\rho_A(\omega)]=\frac{1}{\mathbb{P}(A)},\quad \forall \omega\in A.$$
\end{theorem}
\begin{proof}
	For the proof and more details we refer the reader to \cite{1}.
\end{proof}
\section{The Trading Model: Mean Reversion Timing for $Z_\delta$}\label{sec3}
This section is devoted to study the behavior of the process $Z_\delta$ defined in \ref{xi}. This process consistently reverts to its mean value ($Z_\delta=0$) along any given path \cite{116}. The mean value of $Z_\delta$ is referred to as the reference level. Here, we propose a model to estimate the time intervals during which $Z_\delta$ returns to its mean.
\subsection{The Setup}
Suppose that the price process of a risky asset, $X_t$, is a partially ergodic positive stochastic process \cite{116}. The first recurrence time $\tau_0$ of the associated $Z_\delta$ process to its mean, within a time interval of length $\delta_0$, is defined as:
\begin{align}
	X_t&=X_0e^{Y_t},\quad X_0=x, \quad Z_\delta=\xi_{\delta,W_\delta}^\beta[Y_t],\quad Z_0=0,\label{3-1}\\
	\tau_0&:= \inf \{t\in[0,s] \big\lvert \delta_0=s-t, Z_{\delta_0}=0\}, \quad s>0.\label{3-2}
\end{align}
For more information refer to \cite{1} and \cite{116}. We define the subsequent recurrence times to the reference level as $\tau_i$ for $i\in\{1,2,3,\cdots\}$ such that $\tau_i<\tau_{i+1}$. The difference between each pair of consecutive recurrence times is termed the sojourn time, denoted by $\delta_i$.

 Dividing the path of the $Z_\delta$ based on the reference level zero, we identify two parts: One above and one below the reference level. The set of sojourn times for these levels are denoted by $\phi^+$ and $\phi^-$, respectively:
\begin{align}
	\phi^+&:=\{\delta_i=\tau_{i+1}-\tau_i\big\lvert Z_{\delta_i}>0\},\\
	\phi^-&:=\{\delta_i=\tau_{i+1}-\tau_i\big\lvert Z_{\delta_i}<0\}, \quad \forall i\geq 0.
\end{align}
Hence, the mean sojourn time for the process $Z_\delta$ is defined as follows:
\begin{align*}
	\overline{\phi}^+(z,\delta)&:=\lim_{n\rightarrow \infty}\frac{1}{n}\sum_{k=1}^n\delta_k,\quad \delta_k\in\phi^+,\\
	\underline{\phi}^-(z,\delta)&:=\lim_{n\rightarrow \infty}\frac{1}{n}\sum_{k=1}^n\delta_k,\quad \delta_k\in \phi^-.
\end{align*}
The complete sets of recurrence times and sojourn times are therefore, respectively, defined as: .
\begin{align*}
	\boldsymbol{\tau}&=\{\tau_0,\tau_1,\tau_2,\cdots\},\\
	\boldsymbol{\delta}&=\phi^+\cup\phi^-=\{\delta_0,\delta_1,\delta_2,\cdots\},\quad \delta_i=\tau_{i+1}-\tau_i,\quad \forall i\geq0.
\end{align*}
For each time interval $[\tau_i,\tau_{i+1}]$, let $M_i=\max \lvert Z_{\delta_i}\rvert$.
Denote the time at which the process $\lvert Z_\delta\rvert$ reaches its maximum within this interval as $t_{M_i}$, referred to as the Order Execution Time (OET).

As the process $Z_\delta$ approaches the reference level zero, we initiate a trading position (long, if the process falls below the reference level, short if it rises above). We interpret the time $t_{M_i}$ as the time to exit the trade.   

We form a portfolio, consisting one long and one short positions on a stock. We define the profit of this trade by
\begin{equation}
	\mathcal{V}_t=l\sum_{i\geq 1}\boldsymbol{1}_{[\phi^-]}\lvert X_{\tau_i}-X_{t_{M_i}}\rvert+s\sum_{i\geq 1}\boldsymbol{1}_{[\phi^+]}\lvert X_{\tau_i}-X_{t_{M_i}}\rvert, \label{Vt}
\end{equation}
where  
\begin{equation*}
	\boldsymbol{1}_{[\phi^\cdot]}:=\begin{cases}
		1, &\text{if}\quad\phi^\cdot \neq \emptyset,\\
		0, &\text{if}\quad \phi^\cdot =\emptyset, 
	\end{cases}
\end{equation*}
and $l$ and $s$ are the long and short leverage coefficients, respectively.
\subsection{The Dynamics of Recurrence Times}
Consider the Itô Markov stochastic process
\begin{equation*}
	Y_t=Y_0+\int_0^t \sigma_sdW_s+\int_0^t\mu_sds, \quad Y_0=y.
\end{equation*}
Utilizing the ergodic maker operator (EMO), we get a mean ergodic process $Z_\delta$ as follows:
\begin{align}
	Y_t^\prime&=\ln(X_t)=Y_0^\prime+\int_0^t\sigma_sdW_s+\int_0^t\mu_s ds, \quad Y_0^\prime= y+\ln(x),\notag\\
	Z_\delta&=\xi_{\delta,W_\delta}^\beta[Y_t^\prime]=Z_0+\frac{1}{T^\beta}\int_0^\delta\sigma_sdW_s+\frac{W_T}{T^\beta}\int_0^\delta\mu_sds, \quad Z_0=0,
\end{align}
where $\mu_t$ and $\sigma_t$ are the drift and volatility coefficients, which are adapted integrable functions of $t$ and $X$, $W_t$ is a standard Wiener process, and $\beta$ is the inhibition degree parameter. 
\begin{theorem}\label{th1}
	Consider the time interval $[0,T]$ and let $X_t$ be a partially ergodic stochastic process. Then, the dynamics of the recurrence times process, $\{\tau_i\}_{i\geq0}$, is of the form:
	\begin{align*}
		d\tau_i&=-
		\frac{\big[\sigma_{\tau_i}+\int_0^{\tau_i} \mu_sds\big]}{\mu_{\tau_i}}\frac{dW_{\tau_i}}{W_{\tau_i}}.\\
		\tau_0&= \inf \{t\in[0,s] \big\lvert \delta_0=s-t, Z_{\delta_0}=0\}.
	\end{align*}
\end{theorem}
\begin{proof}
	Consider the fixed path $\omega_0$. When $Z_\delta$ meets its mean along $\omega_0$, at any time interval of length $\delta_i$, we have
	\begin{align*}
		Z_{\delta_i}&=\frac{1}{T^\beta}\int_0^{\delta_i}\sigma_sdW_s+\frac{W_T}{T^\beta}\int_0^{\delta_i}\mu_sds=0
	\end{align*} 
	According to \cite{116} we know that the process $Z_\delta$ returns to its mean in the time interval $[0,T]$ once at least. Therefore, we can consider $\delta_i$ as $\delta_T=T-0=0$. Hence,
	\begin{align*}
		Z_{T}&=\frac{1}{T^\beta}\int_0^{T}\sigma_sdW_s+\frac{W_T}{T^\beta}\int_0^{T}\mu_sds=0.
	\end{align*}
	Using Itô lemma we have
	\begin{align*}
		dZ_{T}=&\big[\frac{-\beta}{T^{\beta+1}}\int_0^{T}\sigma_sdW_s-\frac{\beta W_{T}}{T^{\beta+1}}\int_0^{T}\mu_sds+\frac{\mu_{T}W_{T}}{T^{\beta}}\big]dT \notag\\
		&+\frac{1}{T^{\beta}}\big[\sigma_T+\int_0^T \mu_sds\big]dW_{T}\\
		=&\big[\frac{-\beta}{T}Z_T+\frac{\mu_T W_T}{T^{\beta}}\big]dT+
		\frac{1}{T^{\beta}}\big[\sigma_T+\int_0^T \mu_sds\big]dW_{T}=0.
	\end{align*}
	Since we considered $\delta=T-0=T$, we can write
	\begin{align*}
		dZ_\delta=\big[\frac{-\beta}{\delta}Z_\delta+\frac{\mu_\delta W_\delta}{\delta^{\beta}}\big]d\delta+
		\frac{1}{\delta^{\beta}}\big[\sigma_\delta+\int_0^\delta \mu_sds\big]dW_\delta=0.
	\end{align*}
	When $Z_\delta$ meets its mean, at the time $\tau_i$ in the time interval of length $\delta_i$, we have $Z_{\delta_i}=0$. Take $\delta_i=\tau_i-0=\tau_i$. Therefore, 
	\begin{align*}
		dZ_{\tau_i}=\big[\frac{\mu_{\tau_i} W_{\tau_i}}{\tau_i^{\beta}}\big]d\tau_i+
		\frac{1}{\tau_i^{\beta}}\big[\sigma_{\tau_i}+\int_0^{\tau_i} \mu_sds\big]dW_{\tau_i}=0.
	\end{align*}
	Hence,
	\begin{align*}
		&\big[\frac{\mu_{\tau_i} W_{\tau_i}}{\tau_i^{\beta}}\big]d\tau_i=-
		\frac{1}{\tau_i^{\beta}}\big[\sigma_{\tau_i}+\int_0^{\tau_i} \mu_sds\big]dW_{\tau_i}\\
		&\mu_{\tau_i} W_{\tau_i}d\tau_i=-
		\big[\sigma_{\tau_i}+\int_0^{\tau_i} \mu_sds\big]dW_{\tau_i}\\
		\Rightarrow\quad &d\tau_i=-
		\frac{\big[\sigma_{\tau_i}+\int_0^{\tau_i} \mu_sds\big]}{\mu_{\tau_i}}\frac{dW_{\tau_i}}{W_{\tau_i}}.
	\end{align*}
	Finally, from \ref{3-2} we observe that
	$\tau_0= \inf \{t\in[0,s] \big\lvert \delta_0=s-t, Z_{\delta_0}=0\}$.
\end{proof}
\subsubsection{An Example}
Consider a stock price process $S_t$ following the geometric Brownian motion:
\begin{align}
	S_t&=S_0\exp\{(\mu-\frac{1}{2}\sigma^2)t+\sigma W_t\},\quad S_0=s,\label{s}\\
	Y_t^\prime&=\ln(S_t)=Y_0^\prime+(\mu-\frac{1}{2}\sigma^2)t+\sigma W_t,\quad Y_0^\prime=\ln(s),\notag
\end{align}
where $\mu$ and $\sigma$ are constants, and $W_t$ is a standard Wiener process. Constructing the mean ergodic process $Z_\delta$ yields:
\begin{equation}\label{z}
	Z_\delta=Z_0+\frac{(\mu-\frac{1}{2}\sigma^2)\delta W_T}{T^{\beta}}+\frac{\sigma W_\delta}{T^{\beta}},\quad Z_0=0.
\end{equation}
Applying theorem \ref{th1}, we obtain:
\begin{align*}
	d\tau_i&=-
	\frac{\big[\sigma+\int_0^{\tau_i} (\mu-\frac{1}{2}\sigma^2)ds\big]}{\mu-\frac{1}{2}\sigma^2}\frac{dW_{\tau_i}}{W_{\tau_i}}.\\
	\tau_0&= \inf \{t>0 \big\lvert t\in [t,t+\delta_0] , Z_{\delta_0}=0\}.
\end{align*}
Thus,
\begin{align*}
	d\tau_i&=-
	\frac{\big[\sigma+ (\mu-\frac{1}{2}\sigma^2)\tau_i\big]}{\mu-\frac{1}{2}\sigma^2}\frac{dW_{\tau_i}}{W_{\tau_i}}.\\
	&=-\Big[\frac{\sigma}{\frac{1}{2}\sigma^2-\mu}+\tau_i\Big]\frac{dW_{\tau_i}}{W_{\tau_i}}.
\end{align*}
In the figures below, we display the stock price process (Equation \ref{s}) in figure \ref{fig1}. The process $Z_\delta$ is shown in figure \ref{fig2}, effectively indicates when to take a long or short position on the stock, as illustrated by the two arrows marking two example opportunities.
\begin{figure}[H]
	\begin{center}
		\includegraphics[scale=0.4]{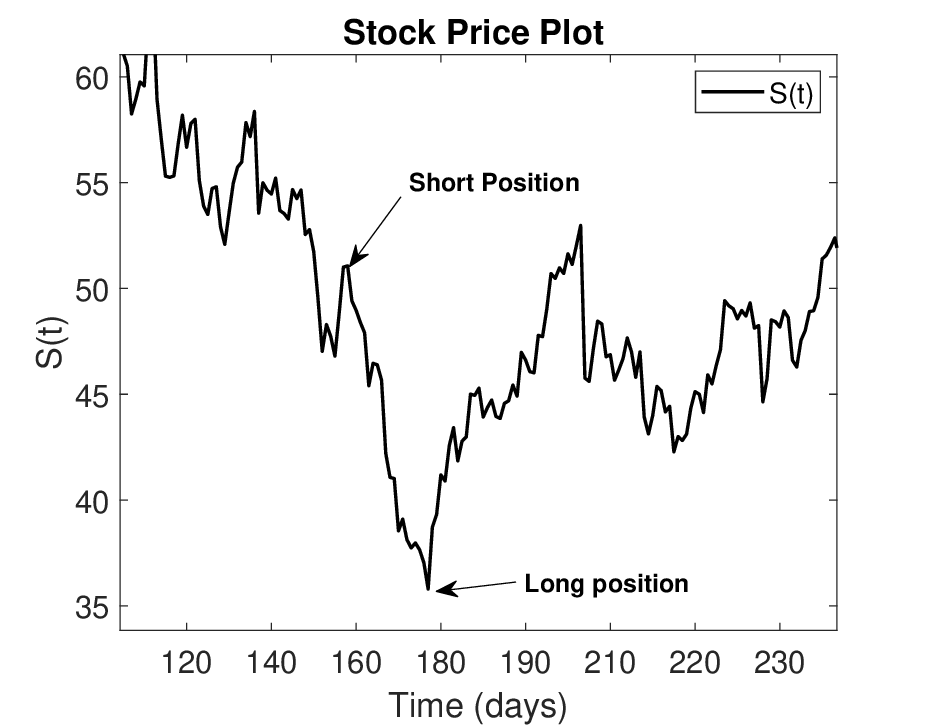}
		\caption{The plot of the stock price process $S_t$.}
		\label{fig1}
	\end{center}
\end{figure}
\begin{figure}[H]
	\begin{center}
		\includegraphics[scale=0.41]{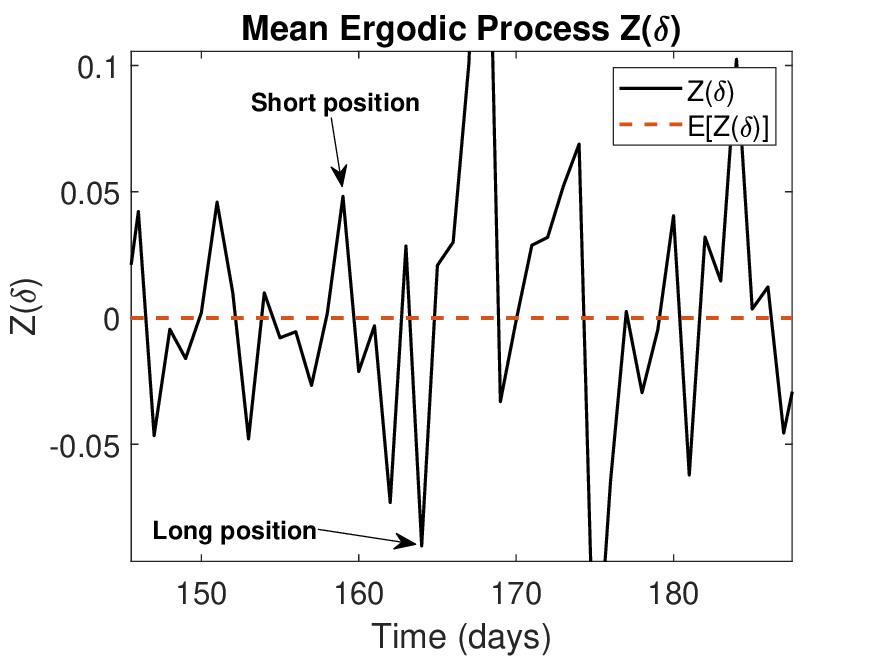}
		\caption{Corresponding mean ergodic $Z_\delta$ process of $S_t$.}
		\label{fig2}
	\end{center}
\end{figure}
Upon comparing the figures, it is apparent that there may be some discrepancies in the plot of $Z_\delta$. Therefore, while the model does not guarantee a flawless strategy, it does help in mitigating trading risk.

\subsection{Estimating $t_{M_i}$}

\begin{theorem}
	The order execution time $t_{M_i}$ satisfies the following relations
	$$\frac{\sigma_{t_{M_i}}}{\mu_{t_{M_i}}}<dt_{M_i}<d\tau_{i+1},\quad i=1,2,3,\cdots$$
\end{theorem}
\begin{proof}
	Consider a fixed path $\omega_0$. Assuming the path of $Z_\delta$ is below the reference level, within any given time interval $[\tau_i,\tau_{i+1}]$ of length $\delta_i$, the absolute value of $Z_\delta(\omega_0)$ lies within $[0,\lvert Z_{t_{M_i}}(\omega_0)\rvert]$ ($\delta_{M_i}=t_{M_i}-0$). Consequently, $Z_{t_{M_i}}(\omega_0)\leq 0$. Therefore,
	\begin{align}
		Z_{t_{M_i}}&=\frac{1}{{t_{M_i}}^\beta}\int_0^{t_{M_i}}\sigma_sdW_s+\frac{W_{t_{M_i}}}{{t_{M_i}}^\beta}\int_0^{t_{M_i}}\mu_sds\leq 0,\notag\\
		& \frac{1}{{t_{M_i}}^\beta}\int_0^{t_{M_i}}\sigma_sdW_s\leq -\frac{W_{t_{M_i}}}{{t_{M_i}}^\beta}\int_0^{t_{M_i}}\mu_sds,\notag\\
		&\int_0^{t_{M_i}}\sigma_sdW_s\leq W_{t_{M_i}}\int_0^{t_{M_i}}\mu_sds,\quad\Rightarrow \quad \frac{\int_0^{t_{M_i}}\sigma_sdW_s}{\int_0^{t_{M_i}}\mu_sds}\leq W_{t_{M_i}},\notag\\
		&\frac{\sigma_{t_{M_i}}dW_{t_{M_i}}}{\mu_{t_{M_i}}dt_{M_i}}\leq dW_{t_{M_i}},\quad \Rightarrow \quad dt_{M_i} \geq \frac{\sigma_{t_{M_i}}}{\mu_{t_{M_i}}}.\label{I}
	\end{align}
	Furthermore, considering $\delta_i=\tau_{i+1}-\tau_i$ 
	\begin{align}
		&\tau_i<t_{M_i}<\tau_{i+1},\quad \Rightarrow \quad 0<t_{M_i}-\tau_i<\tau_{i+1}-\tau_{i},\notag\\
		\Rightarrow& 0< t_{M_i}-\tau_i<\delta_i.\label{II}
	\end{align}
	Hence, from \ref{I}
	\begin{align*}
		&dt_{M_i} \geq \frac{\sigma_{t_{M_i}}}{\mu_{t_{M_i}}},\\
		&dt_{M_i}-d\tau_i\geq \frac{\sigma_{t_{M_i}}}{\mu_{t_{M_i}}}- d\tau_i.
	\end{align*}
	Using \ref{II} yields:
	\begin{align*}
		&\frac{\sigma_{t_{M_i}}}{\mu_{t_{M_i}}}- d\tau_i\leq dt_{M_i}-d\tau_i <d\delta_i=d\tau_{i+1}-d\tau_i\\
		&\frac{\sigma_{t_{M_i}}}{\mu_{t_{M_i}}}\leq dt_{M_i}<d\tau_{i+1}.
	\end{align*}
\end{proof}
\section{Irrational Rotation}\label{sec4}
%\begin{definition}
%	An irrational rotation is a map given by
%	\begin{equation}
%		R_{\theta}:[0,1]\rightarrow [0,1], \quad R_{\theta}=x+\theta \quad \mod 1,
%	\end{equation}
%	where $\theta$ is an irrational number.
%\end{definition}
The circle rotation can be visualized as a subdivision of a circle into two parts, which are then exchanged. When subdivided into more than two parts that are permuted with each other, it is called an interval exchange transformation \cite{1}.

There is a notation for the irrational rotation known as the multiplicative notation.
\begin{definition}(Irrational rotation)
	For the unit circle $\mathbb{S}^1$, let
	\begin{equation}
		R_\theta:\mathbb{S}^1\rightarrow \mathbb{S}^1,\quad R_\theta(x)=xe^{2\pi i \theta},\label{IR}
	\end{equation}
	where $\theta$ is an irrational number. We call the mapping \ref{IR} an irrational rotation on the unit circle.
\end{definition}
\subsection{Stochastic Process $\theta_t$}
We consider $\theta_t$ as a stochastic process dependent on the price process of a risky asset, $X_t$. For all $t\geq 0$:
$$0<\theta_t\neq \frac{p}{q}<n,\quad n\in \mathbb{N}\backslash \{\infty\}.$$
Therefore,
\begin{equation}\label{4-3}
	0<\theta_t^2<n^2\Rightarrow 0<\mathbb{E}[\theta_t^2]<n^2.
\end{equation}
For all $k\in \mathbb{N}$, the angles zero and $2k\pi$ overlap. However, we do not consider these angles to be the same.

Assume the price process of an asset follows \ref{3-1}. The European call option price with strike price $K$ and maturity $T$ is
$$C(t,X_t)=e^{-r(T-t)}\mathbb{E}^Q\big[\max[xe^{Y_t}-K,0]\big],$$
We have:
\begin{align*}
	\ln(xe^{Y_t}-K)&=\ln(xe^{Y_t})+\ln(1-\frac{K}{xe^{Y_t}})\\
	&=\ln(x)+Y_t+\ln(1-\frac{K}{xe^{Y_t}}).
\end{align*}	
Using the ergodic maker operator (EMO) yields:
\begin{align}
	\xi_{\delta,W_\delta}^\beta[\ln(xe^{Y_t}-K)]=Z_\delta+\xi_{\delta,W_\delta}^\beta\Big[\ln(1-\frac{K}{xe^{Y_t}})\Big]\label{xiln}.
\end{align}
\begin{lemma}
	The number, $\ln(1-\frac{K}{xe^{Y_t}})$ is irrational.
\end{lemma}	
\begin{proof}
	For a  European call option to be exercised, we need to have $xe^{Y_t}>K$. Hence,
	$0<1-K/xe^{Y_t}<1$. From the Lindemann-Weierstrass theorem \cite{100}, it follows that $e^a$ is non-algebraic, for every positive non-algebraic number $a$. Specifically, if $a$ is rational, the $e^a$ cannot be rational. Therefore, $\ln(1-K/xe^{Y_t})$ is an irrational number, according to \cite{99,100}. 
\end{proof}
Now since \ref{xiln} is irrational, we define the process $\theta(z,\delta)$ as follows:
\begin{definition}
	For $\delta>0$, the process
	\begin{equation}
		\theta(z,\delta)=Z_\delta+\frac{W_T}{T^{\beta}}\gamma_\delta,
	\end{equation}
	with $\gamma_\delta=\ln\big(1-\frac{K}{xe^{Y_\delta}}\big)$ is called the irrational angle process. We denote the irrational rotation generated using $\theta(z,\delta)$ by $R_{\theta_t}(\cdot)$.
\end{definition}
\begin{proposition}\label{prop1}
	The process $\theta(z,\delta)$ is a wide-sense stationary process.
\end{proposition}
\begin{proof}
	First, we evaluate the expectation of the process.
	\begin{align*}
		&\mathbb{E}[\theta(z,\delta)]=\mathbb{E}[Z_\delta+\frac{W_T}{T^{\beta}}\gamma_\delta]=\mathbb{E}[Z_\delta]+\frac{\mathbb{E}[W_T]\gamma_\delta}{T^{\beta}}=0.
	\end{align*}
	Additionally, we have:
	\begin{align*}
		&\mathbb{V}ar[\theta(z,\delta)]=\mathbb{E}[\theta(z,\delta)^2]-\mathbb{E}[\theta(z,\delta)]^2=\mathbb{E}[Z_\delta^2+\frac{W_\delta^2}{\delta^{2\beta}}\gamma_\delta^2]\\
		=&\mathbb{E}[Z_\delta^2]+\frac{\gamma_\delta^2}{T^{2\beta-1}}
		=\frac{1}{T^{2\beta}}\int_0^\delta\sigma_s^2ds+\frac{1}{T^{2\beta-1}}(\int_0^\delta\mu_sds)^2+\frac{\gamma_\delta^2}{T^{2\beta-1}}\\
		\Rightarrow &\mathbb{V}ar[\theta(z,\delta)]	=\frac{1}{T^{2\beta}}\Big[\int_0^\delta\sigma_s^2ds+t\big[ (\int_0^\delta\mu_sds)^2+\gamma_\delta^2\big]\Big]
	\end{align*}
	Since $Y_t$ is an Itô process:
	$\int_0^\delta (\sigma_s^2+\lvert\mu_s\rvert)ds<\infty$. Hence, $\mathbb{E}[\theta^2(z,\delta)]<\infty$.\\
	For all $\delta,\epsilon>0$:
	$$\mathbb{E}[\theta(z,\delta)]=\mathbb{E}[\theta(z,\delta+\epsilon)]=0.$$
	Since the process $\theta(z,\delta)$ is dependent only on $\delta$, the correlation function of $\theta(z,\delta)$ is also a function of $\delta$. Therefore, $\theta(z,\delta)$ is wide-sense stationary.
\end{proof}
\begin{proposition}\label{prop2}
	The process $\theta(z,\delta)$ is mean ergodic.
\end{proposition}
\begin{proof}
	Since $Z_\delta$ is mean ergodic and 
	$\xi_{\delta,W_\delta}^\beta[\gamma_\delta]$ is mean ergodic \cite{116}, from the properties of mean ergodic processes, $\theta(z,\delta)$ is mean ergodic.
\end{proof}
\subsection{Properties of the Irrational Rotation}
\begin{proposition}\label{4-8}
	If $R_{\theta}$ is an irrational rotation on the unit circle, with $\theta$ being an irrational number, then,
	\begin{itemize}
		\item [1.] The orbit of every $x\in[0,1]$ under $R_\theta$ is dense in the interval $[0,1]$.
		\item[2.] 
		$R_\theta$ is not topologically mixing. 
		\item[3.] 
		$R_{\theta}$ is uniquely ergodic, with the lebesgue measure as the unique invariant probability measure.
		\item[4.]
		Let $[a,b]\in[0,1]$. From the ergodicity of $R_\theta$ we have
		$$\lim_{N\rightarrow \infty}\sum_{n=0}^{N-1}\chi_{[a,b)}(R_\theta^n(x))=b-a,\quad x\in[a,b].$$
	\end{itemize} 
\end{proposition}
\begin{proof}
	For the proof and more details we refer the reader to \cite{1}.
\end{proof}
\begin{theorem}
	The process $R_{\theta_t}$ is Markov.
\end{theorem}
\begin{proof}
	It suffices to prove the increments of $R_{\theta_t}$ are independent. We have
	$$ R_{\theta_t}(0)=0\cdot e^{2\pi i \theta_t}=0, \quad \forall t>0.$$
	According to the proposition \ref{4-8}, since the irrational rotation is ergodic and the orbit of every point is dense, it follows that for the time $T$ and every $s,t\geq0$, the angles $\theta_t$ and $\theta_s$ are independent from each other. Therefore, $R_{\theta_s}$ and $R_{\theta_t}$ are independent. Hence,
	\begin{align*}
		&\mathbb{P}(R_{\theta_t}\leq x\lvert R_{\theta_{t_1}},\cdots,R_{\theta_{t_n}})\\
		=&\mathbb{P}(R_{\theta_t}-R_{\theta_{t_n}}+R_{\theta_{t_n}}\leq x\lvert R_{\theta_{t_1}},\cdots,R_{\theta_{t_n}})\\
		=&\mathbb{P}(R_{\theta_t}-R_{\theta_{t_n}}+R_{\theta_{t_n}}\leq x\lvert R_{\theta_{t_n}})=\mathbb{P}(R_{\theta_t}\leq x\lvert R_{\theta_{t_n}})
	\end{align*}
\end{proof}
We express the following theorem by using \cite{101} to define the function $\phi$, which is used in the Birkhoff ergodic theorem.
\begin{theorem}\label{th4-9}
	Let $R_{\theta_t}$ be an irrational rotation process and $\phi:[0,1]\rightarrow \mathbb{R}$ be a continuous function, such that $\phi(0)=\phi(1)$. Then,
	$$\lim_{n\rightarrow \infty}\Big(\frac{1}{n}\sum_{k=0}^{n-1}\phi(R_{\theta_t}^k(x))\Big)=\int_0^1\phi(y)dy,\quad \forall x\in [0,1].$$
\end{theorem}
\begin{proof}
	From \cite{101} we define the function $\psi_m$ as
	$$\psi_m(x)=e^{2\pi i mx}=\cos(2\pi i m x)+i\sin(2\pi i mx).$$
	We have
	$$	\psi_m(R_{\theta_t}^k(x))=e^{2\pi i k\theta_t+x}=e^{2\pi i m k \theta_t}\psi_m(x).$$
	For $m\neq 0$ we write
	\begin{align*}
		\Big\lvert\frac{1}{n}\sum_{k=0}^{n-1}\psi_m(R_{\theta_t}^k(x))\Big\rvert&=\frac{1}{n}\cdot \lvert e^{2\pi i mx}\rvert\cdot \Big\lvert\sum_{k=0}^{n-1}e^{2\pi i m k \theta_t}\Big\rvert\\
		&=\frac{1}{n}\Big\lvert\frac{1-e^{2\pi i n m \theta_t}}{1-e^{2\pi i m \theta_t}}\Big\rvert\leq\frac{1}{n}\cdot \frac{2}{1-e^{2\pi i m \theta_t}}\xrightarrow[n\rightarrow \infty]{} 0.
	\end{align*}
	Therefore, if we take $\phi(x)=\sum_{m=-t_N}^{t_N}\theta_m\psi_m(x)$, in which\\ $\theta_{-t_N},\theta_{-t_{N+1}},\cdots, \theta_{t_0} ,\cdots, \theta_{t_N}\in \mathbb{Q}^{\complement}$, Then,
	\begin{equation*}
		\lim_{n\rightarrow \infty}\frac{1}{n}\sum_{k=0}^{n-1}\phi(R_{\theta_t}^k(x))=\theta_{t_0}=\int\phi(y)dy.
	\end{equation*}
	Finally, from \cite{1} we know that the set of trigonometric polynomials is dense in the space of all periodic functions, which completes the proof.
\end{proof}
\begin{remark}
	In theorem \ref{th4-9} the angle $\theta_{-i}$, for $i>0$, indicates clockwise direction on the unit circle. Also $t_0$ indicates the time that we buy an option contract. Therefore, the angle $\theta_0=\theta_{t_0}$ is not zero and it varies with respect to $Z_{\delta_0}$. Hence, $\theta_{t_0}$ is a random irrational number. 
\end{remark}
\begin{theorem}
	Let $\theta(z,\delta)$ be a random angle process with respect to irrational rotation $R_{\theta_t}$ on $\mathbb{S}^1$ such that $\mathbb{E}[\theta_t]<\infty$. Then, $R_{\theta_t}=xe^{2\pi i \theta_t}$ is a partially ergodic process.
\end{theorem}
\begin{proof}
	We have
	\begin{align*}
		Y_t=\ln(R_{\theta_t})=\ln(x)+2\pi i \theta(z,t),\quad Y_0=\ln(x).
	\end{align*}
	Let $\theta(z,\delta)=\theta_\delta$. Using EMO yields
	\begin{align*}
		Z_\delta=\xi_{\delta,W_\delta}^\beta [Y_t]=0+\frac{2\pi i W_T}{T^\beta}\theta_\delta,\quad Z_0=0.
	\end{align*}
	Evaluating the covariance of $Z_\delta$ yields
	\begin{align}\label{cc}
		\mathbf{Cov}_{zz}=\mathbb{E}[Z_\delta^2]=-\frac{4\pi^2}{T^{\beta-1}}\mathbb{E}[\theta_\delta^2].
	\end{align}
	Now, it suffices to prove
	\begin{align}\label{bb}
		<Z>=\lim_{T\rightarrow\infty}\frac{1}{T}\int_0^T(1-\frac{\delta}{T})\mathbf{Cov}_{zz}(\delta)d\delta=0.
	\end{align}
	Substituting \ref{cc} in \ref{bb} yields
	\begin{align}
		<Z>&=\lim_{T\rightarrow\infty}\frac{1}{T}\int_0^T(\frac{\delta}{T}-1)\frac{4\pi^2}{T^{\beta-1}}\mathbb{E}[\theta_\delta^2]d\delta\notag\\
		&=\lim_{T\rightarrow\infty}\frac{1}{T}\int_0^T\frac{\delta}{T}\frac{4\pi^2}{T^{\beta-1}}\mathbb{E}[\theta_\delta^2]d\delta-\lim_{T\rightarrow\infty}\frac{1}{T}\int_0^T\frac{4\pi^2}{T^{\beta-1}}\mathbb{E}[\theta_\delta^2]d\delta\notag\\
		&=\lim_{T\rightarrow\infty}\frac{4\pi^2}{T^{\beta+1}}\int_0^T\delta\mathbb{E}[\theta_\delta^2]d\delta-\lim_{T\rightarrow\infty}\frac{4\pi^2}{T^{\beta}}\int_0^T\mathbb{E}[\theta_\delta^2]d\delta\label{ints}.
	\end{align}
	From \ref{4-3} we know that $\mathbb{E}[\theta_\delta^2]=m<\infty$. Therefore, the integrals in \ref{ints} are finite. Hence, taking the limit as $T\rightarrow \infty$ yields $<Z>=0$. 
\end{proof}
 Figure \ref{fig3} demonstrates how the value of $R_{\theta_t}$ varies with respect to stock price.
\begin{figure}[H]
	\begin{center}
		\includegraphics[scale=0.4]{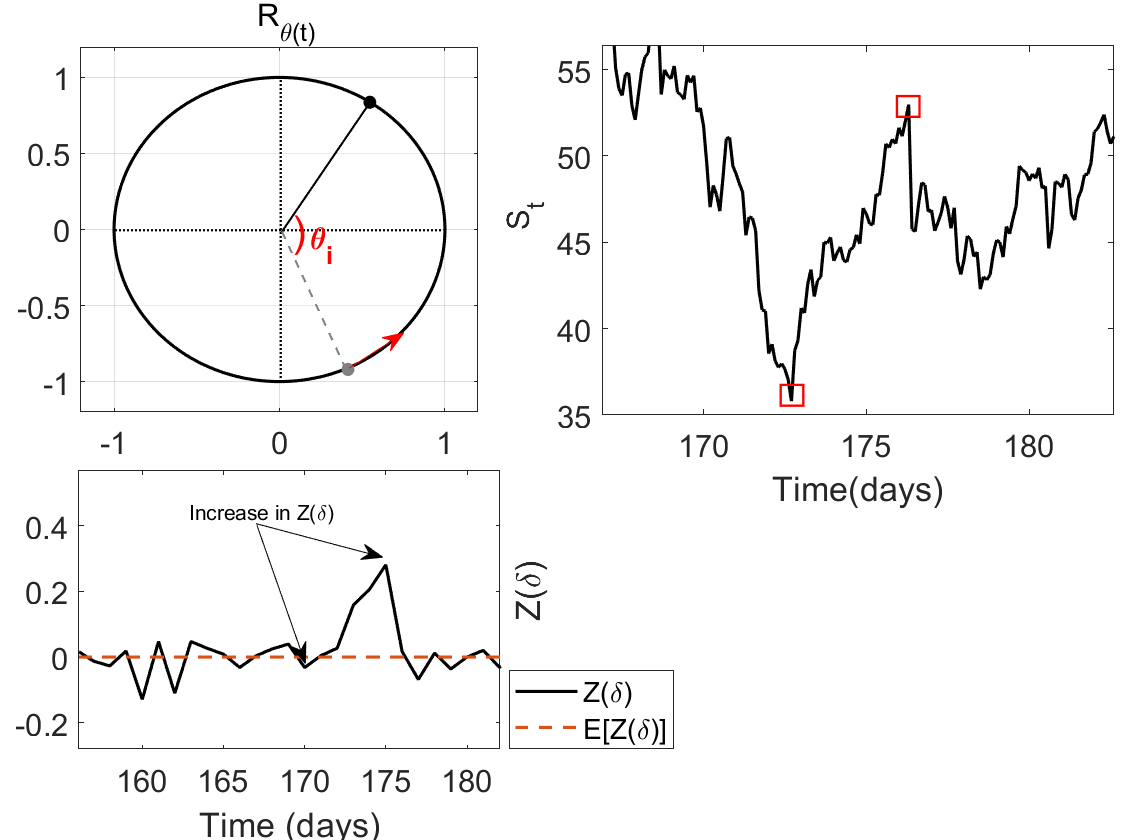}
		\caption{The behavior of the process $R_{\theta_t}$ on the unit circle relative to the stock price.}\label{fig3}
	\end{center}
\end{figure}

\section{European Call Option Pricing Using The Irrational Rotation}\label{sec5}
In order to study the option pricing problem using the irrational rotation, it is necessary to consider the $K^\prime $ in $[0,1]$ that corresponds to the strike price $K$. We therefore denote this correspondence by $\phi$ and write $\phi(K^\prime)=K$.
\begin{theorem}
The European call option price with strike price $K$ and the exercise time $T$ is given by
$$C(R_{\theta_t},K)=e^{-rt}\Big(\frac{W_{T}}{T^\beta}\ln(1-\frac{K}{S_{t_0}})-K\Big),$$
where $R_{\theta_t}$ is the irrational rotation on the unit circle $\mathbb{S}^1$.
\end{theorem}
\begin{proof}
From \cite{13} we write
\begin{align}
	C(R_{\theta_t},K)=&e^{-rt}\mathbb{E}^Q[\phi(R_{\theta_t}(x))-\phi(K^\prime)]\notag\\
	=&e^{-rt}\mathbb{E}^Q[\phi(xe^{2\pi i\theta_t})-K]\notag\\
	=&e^{-rt}\big[\mathbb{E}^Q[\phi (xe^{2\pi i \theta_t})]-K\big]\notag\\
	=&e^{-rt}\Big(\lim_{n\rightarrow \infty}\frac{1}{n}\sum_{k=0}^{n-1}\phi(R_{\theta_t}^k(x))-K\Big)\notag\\
	=&e^{-rt}\Big(\int \phi(y)dy-K\Big).\label{88}
\end{align}
From \ref{th4-9} we have
\begin{align*}
	\phi(y=x)&=\sum_{m=-{t_N}}^{t_N}\theta_m\psi_m(x),\\
	\psi_m(x)&=e^{2\pi i mx}=\cos(2\pi i m x)+i\sin(2\pi i mx).
\end{align*}
Therefore,
\begin{align}
	\int \phi(y)dy&= \int \theta_0 \psi_0(y)dy\notag\\
	&=\int_0^1 \theta_0 dy=\theta_{t_0}=\frac{W_{T}}{T^{\beta}}\ln(1-\frac{K}{S_{t_0}}).\label{92}
\end{align}
Substituting \ref{92} in \ref{88} yields
\begin{align*}
	C(R_{\theta_t},K)=e^{-rt}\Big(\frac{W_{T}}{T^{\beta}}\ln(1-\frac{K}{S_{t_0}})-K\Big),
\end{align*}
where $S_{t_0}$ is the stock price at the time we buy the option contract.
\end{proof}
\section{Black–Scholes Ergodic Partial Differential Equation}\label{sec6}
From \cite{116} we have the following proposition.
\begin{proposition}(Black–Scholes ergodic Partial differential equation)
Under the assumptions of the Black-Scholes model, the European call option price, $C(Z_\delta,\delta)$, relative to the stock price variation $Z_\delta=z$, with respect to short rate $r$, inhibition degree parameter $\beta$, and the strike price $K$ satisfies in the following partial differential equation.
\begin{align}
	&\frac{\partial C}{\partial \delta}+rz\frac{\partial C}{\partial z}+\frac{1}{2}B_\delta^2\frac{\partial^2 C}{\partial z^2}-rC=0,\label{BSP}\\
	&\text{for} \quad 0<\lvert z\rvert<\infty,\quad 0<\delta<\delta_T=T-0,\notag\\
	&\text{where}\quad B_\delta=\frac{q}{\delta^{\beta-1}}+\frac{\sigma}{\delta^{\beta}}, \quad q=\mu-\frac{1}{2}\sigma^2,\notag
\end{align}
together with initial conditions $C(0,\delta)=0$ and $C(z,\delta_T)=(\lvert z\rvert-\ln(K))^+$.
\end{proposition}
In this section our goal is to solve this problem.
\begin{proposition}\label{p1}
The PDE \ref{BSP} has the representation of the form
\begin{align}\label{garm}
	&\frac{\partial U}{\partial \tau}=\frac{1}{2}\eta\frac{\partial^2 U}{\partial y^2},
\end{align}
where $0<\tau<T$, $-\infty<y<\infty$, and $\eta=B_\delta^2T^{2\beta}$.
\end{proposition}
\begin{proof}
From \cite{116} we have
\begin{align*}
	Z_\delta&=\frac{1}{T^\beta}\Big[(\mu-\frac{1}{2}\sigma^2)\delta W_T+\sigma W_\delta\Big]\\
	Z_\delta&=\frac{1}{T^\beta}\Big[(\mu-\frac{1}{2}\sigma^2)\delta-(\mu-\frac{1}{2}\sigma^2)\delta+(\mu-\frac{1}{2}\sigma^2)\delta W_T+\sigma W_\delta\Big]\\
	T^\beta Z_\delta&=Y_\delta+(\mu-\frac{1}{2}\sigma^2)[W_T-\delta]\\
	Y_\delta&=T^\beta Z_\delta-(\mu-\frac{1}{2}\sigma^2)[W_T-\delta]
\end{align*}
Therefore, we consider the following change of variables.
\begin{align}\delta=T-\tau,\quad y= T^\beta z-(\mu-\frac{1}{2}\sigma^2)[W_T-\delta]\label{tagh}
\end{align}
Using \ref{tagh} we have
\begin{align}
	&\frac{\partial C}{\partial \delta}=\frac{\partial C}{\partial \tau}\frac{\partial \tau}{\partial \delta}=-\frac{\partial C}{\partial \tau},\label{i}\\
	&\frac{\partial C}{\partial z}=\frac{\partial C}{\partial y}\frac{\partial y}{\partial z}=T^\beta\frac{\partial C}{\partial y},\label{ii}\\
	&\frac{\partial^2 C}{\partial z^2}=\frac{\partial }{\partial z}\big(\frac{\partial C}{\partial z}\big)=\frac{\partial }{\partial z}\big(T^\beta\frac{\partial C}{\partial y}\big)=\frac{\partial }{\partial y}\big(T^\beta\frac{\partial C}{\partial y}\big)\frac{\partial y}{\partial z}\notag\\
	=&\big(\frac{\partial^2 C}{\partial y^2}-\frac{\partial C}{\partial y}\big)T^{2\beta}.\label{iii}
\end{align}
Substituting \ref{i}, \ref{ii}, and \ref{iii} in \ref{BSP} yields
\begin{align}\label{BS2}
	-\frac{\partial C}{\partial \tau}+rzT^\beta\frac{\partial C}{\partial y}+\frac{1}{2}B_\delta^2\big(\frac{\partial^2 C}{\partial y^2}-\frac{\partial C}{\partial y}\big)T^{2\beta}-rC=0
\end{align}
Now take
\begin{align}
	&C(y,\tau)=U(y,\tau)e^{ay+b\tau},\label{u}\\
	&\frac{\partial C}{\partial \tau}=(bU+\frac{\partial U}{\partial \tau})e^{ay+b\tau},\notag\\
	&\frac{\partial C}{\partial y}=(aU+\frac{\partial U}{\partial y})e^{ay+b\tau},\notag\\
	&\frac{\partial^2 C}{\partial y^2}=[a^2U+2a\frac{\partial U}{\partial y}+\frac{\partial^2 U}{\partial y^2}]e^{ay+b\tau}.\notag
\end{align}
Substituting the above relations in \ref{BS2} yields
\begin{align*}	
	bU+\frac{\partial U}{\partial\tau}=&{rzT^\beta}\big[aU+\frac{\partial U}{\partial y}\big]+\frac{B_\delta^2T^{2\beta}}{2}\big[(a^2-a)U
	+(2a-1)\frac{\partial U}{\partial y}+\frac{\partial^2 U}{\partial y^2}\big]-rU,\\
	\Rightarrow \frac{\partial U}{\partial \tau}=&\big[rzT^\beta+\frac{B_\delta^2T^{2\beta}}{2}(2a-1)\big]\frac{\partial U}{\partial y}+\frac{B_\delta^2T^{2\beta}}{2}\frac{\partial^2 U}{\partial y^2}\\
	&+\big[rzT^\beta a+\frac{B_\delta^2T^{2\beta}}{2}(a^2-a)-r-b\big]U.	
\end{align*}
Now take 
$$a=\frac{1}{2}-\frac{rz}{B_\delta^2T^\beta},\quad \text{and} \quad b=\frac{4rzT^{\beta}-B_\delta^2T^{2\beta}}{8}-\frac{r^2z^2}{2B_\delta^2}-r.$$
Therefore, we have
\begin{equation}
	\frac{\partial U}{\partial \tau}=\frac{1}{2}\underbrace{B_\delta^2T^{2\beta}}_{\eta}\frac{\partial^2 U}{\partial y^2},\label{H1}
\end{equation}
With initial condition
$$U(y,0)=C(z,T)e^{-ay}=\max[\lvert z\rvert-\ln(K)].$$
\end{proof}
\subsection{Solving the Heat Equation}
\begin{proposition}\label{p2}
The heat equation \ref{garm} has a solution of the form
\begin{align*}
	U(y,\tau)= \int_{-\infty}^\infty K(y-\gamma,\tau)f(\gamma)d\gamma.
\end{align*}
\end{proposition}
\begin{proof}
Using Fourier transform we have
\begin{align*}
	&\mathcal{\boldsymbol{F}}\{U(y,\tau)\}=\hat{U}(\omega,\tau)=\frac{1}{\sqrt{2\pi}}\int_{-\infty}^\infty U(y,\tau)e^{-i\omega y}dy\\
	\Rightarrow &\mathcal{\boldsymbol{F}}\{U_\tau(y,\tau)\}=\frac{\partial \hat{U}(\omega,\tau)}{\partial \tau}=\hat{U}_\tau(\omega,\tau),\\
	&\mathcal{\boldsymbol{F}}\{U_{yy}(y,\tau)\}=-\omega^2\mathcal{\boldsymbol{F}}\{U(y,\tau)\}=-\omega^2\hat{U}(\omega,\tau).
\end{align*}
From \ref{H1} we have
\begin{align}
	&\mathcal{\boldsymbol{F}}\{U_\tau\}=\mathcal{\boldsymbol{F}}\Big\{\frac{B_\delta^2T^{2\beta}}{2}U_{yy}\Big\}\notag\\
	&\hat{U}_\tau(\omega,\tau)=-\frac{B_\delta^2T^{2\beta}}{2}\omega^2\hat{U}(\omega,\tau).\label{d}
\end{align}
The equation \ref{d} is partial differential equation with solution
\begin{align*}
	\hat{U}(\omega,\tau)=c(\omega)\exp\Big\{-\frac{B_\delta^2T^{2\beta}}{2}\omega^2\tau\Big\}.
\end{align*}
Using Fourier transform for the initial condition $U(y,0)=f(x)$ yields
\begin{align*}
	\mathcal{\boldsymbol{F}}\{U(y,0)\}=\hat{U}(\omega,0)=\hat{f}(\omega),\quad \hat{U}(\omega,0)=c(\omega).
\end{align*}
Therefore,
\begin{align*}
	\hat{U}(\omega,\tau)&=\hat{U}(\omega,0)\exp\Big\{-\frac{B_\delta^2T^{2\beta}}{2}\omega^2\tau\Big\}\\
	&=\hat{f}(\omega)\exp\Big\{-\frac{B_\delta^2T^{2\beta}}{2}\omega^2\tau\Big\}.
\end{align*}
Now using inverse Fourier transform we have
\begin{align*}
	U(y,\tau)&=\mathcal{\boldsymbol{F}}^{-1}\{\hat{U}(\omega,\tau)\}=\frac{1}{\sqrt{2\pi}}\int_{-\infty}^{\infty}\hat{U}(\omega,\tau)e^{i\omega y}dy\\
	&=\frac{1}{\sqrt{2\pi}}\int_{-\infty}^{\infty}\hat{f}(\omega)e^{-\frac{B_\delta^2T^{2\beta}}{2}\omega^2\tau}e^{i\omega y}d\omega\\
	&=\frac{1}{\sqrt{2\pi}}\int_{-\infty}^{\infty}\Big(\frac{1}{\sqrt{2\pi}}\int_{-\infty}^\infty f(\gamma)e^{-i\omega\gamma}d\gamma \Big)e^{i\omega y-\frac{B_\delta^2T^{2\beta}}{2}\omega^2\tau}d\omega\\
	&=\int_{-\infty}^\infty \Big(\frac{1}{2\pi}\int_{-\infty}^\infty e^{i\omega(y-\gamma)-\frac{B_\delta^2T^{2\beta}}{2}\omega^2\tau}d\omega\Big)f(\gamma)d\gamma.
\end{align*}
Assuming
\begin{equation}\label{K}
	K(y-\gamma,\tau)=\frac{1}{2\pi}\int_{-\infty}^\infty e^{i\omega(y-\gamma)-\frac{B_\delta^2T^{2\beta}}{2}\omega^2\tau}d\omega.
\end{equation}
yields
\begin{equation*}
	U(y,\tau)= \int_{-\infty}^\infty K(y-\gamma,\tau)f(\gamma)d\gamma.
\end{equation*}
\end{proof}
\begin{corollary}\label{l1}
If $\gamma=0$ in the theorem above, then we have:
\begin{equation*}
	K(y,\tau)=\frac{1}{\sqrt{2\pi\tau B_\delta^2T^{2\beta}}}e^{-\frac{y^2}{2\tau B_\delta^2T^{2\beta}}}.
\end{equation*}
\end{corollary}
\begin{proof}
Taking $\gamma=0$ in \ref{K} yields
\begin{align*}
	K(y,\tau)&=\frac{1}{2\pi}\int_{-\infty}^\infty e^{i\omega y-\frac{B_\delta^2T^{2\beta}}{2}\omega^2\tau}d\omega\\
	&=\frac{1}{2\pi}\int_{-\infty}^\infty e^{-\frac{y^2}{2\tau B_\delta^2T^{2\beta}}} e^{\big(\frac{\sqrt{B_\delta^2T^{2\beta}}}{\sqrt{2}}\omega\sqrt{\tau}-\frac{iy}{\sqrt{2\tau B_\delta^2T^{2\beta}}}\big)^2}d\omega.
\end{align*}
Let
\begin{align*}
	\lambda=\frac{\sqrt{B_\delta^2T^{2\beta}}}{\sqrt{2}}\omega\sqrt{\tau}-\frac{iy}{\sqrt{2\tau B_\delta^2T^{2\beta}}}
	\Rightarrow d\lambda=\sqrt{\frac{\tau B_\delta^2T^{2\beta}}{2}}d\omega.
\end{align*}
Since $\int_{-\infty}^{\infty}e^{-\lambda^2}d\lambda=\sqrt{\pi}$,
\begin{align*}
	K(y,\tau)&=\frac{1}{2\pi}e^{-\frac{y^2}{2\tau B_\delta^2T^{2\beta}}}\int_{-\infty}^{\infty}e^{-\lambda^2}\frac{\sqrt{2}}{\sqrt{\tau B_\delta^2T^{2\beta}}}d\lambda\\
	&=\frac{1}{2\pi}\frac{\sqrt{2}}{\sqrt{\tau B_\delta^2T^{2\beta}}}e^{-\frac{y^2}{2\tau B_\delta^2T^{2\beta}}}\int_{-\infty}^\infty e^{-\lambda^2}d\lambda\\
	&=\frac{1}{\sqrt{2\pi\tau B_\delta^2T^{2\beta}}}e^{-\frac{y^2}{2\tau B_\delta^2T^{2\beta}}}.
\end{align*}
\end{proof}
\subsection{European Call Option Price: Solution of the Equation \ref{BSP}}
\begin{theorem}
The European call option price with respect to underlying asset price $X$, Expiration time $T$, price variations $z$, the EMO $\xi_{\delta,W_\delta}^\beta[\cdot]$, and the strike price $K$ satisfies the following equation. 
\begin{align*}
	C(z,\tau)&=e^{-r\tau}e^{\frac{y(\lambda -2)+\frac{1}{4}p(\lambda-2)^2}{2\lambda}}\big[\lvert z\rvert-\ln(K)\big]N[d],
\end{align*}
where
$$d=\frac{\ln\Big[X/\ln(K)\Big]}{\sqrt{2p\lambda}},\quad p=r\lvert z\rvert\tau T^{\beta},\quad \lambda=\frac{T^\beta B_\delta^2}{r\lvert z\rvert}.$$
\end{theorem}
\begin{proof}
Using propositions \ref{p1}, \ref{p2}, lemma \ref{l1}, and \ref{u}
we write
\begin{align*}
	e^{ay+b\tau}=\exp\Big\{\frac{\tau B_\delta^2T^{2\beta}y-2\tau T^\beta r\lvert z\rvert y+r\lvert z\rvert\tau^2B_\delta^2T^{3\beta}-\frac{1}{4}B_\delta^4T^{4\beta}\tau^2-r^2z^2\tau^2T^{2\beta}}{2\tau B_\delta^2T^{2\beta}}\Big\}.
\end{align*}
Let $p=r\lvert z\rvert\tau T^{\beta}$, Then
\begin{align*}
	e^{ay+b\tau}&=\exp\Big\{\frac{\tau B_\delta^2 T^{2\beta}y-2py+p\tau B_\delta^2T^{2\beta}-\frac{1}{4}B_\delta^4T^{4\beta}\tau^2-p^2}{2\tau T^{2\beta}B_\delta^2}\Big\}\\
	&=\exp\Big\{\frac{B_\delta^2 T^{\beta}\frac{p}{r\lvert z\rvert}y-2py+p B_\delta^2T^{\beta}\frac{p}{r\lvert z\rvert}-\frac{1}{4}B_\delta^4T^{2\beta}\frac{p^2}{r^2z^2}-p^2}{2\frac{p}{r\lvert z\rvert}T^{\beta}B_\delta^2}\Big\}\\
	&=\exp\Big\{\frac{p\big[\frac{B_\delta^2 T^{\beta}}{r\lvert z\rvert}y-2y\big]+p^2\big[\frac{ B_\delta^2T^{\beta}}{r\lvert z\rvert}-\frac{1}{4}\frac{B_\delta^4T^{2\beta}}{r^2z^2}-1\big]}{2p\frac{T^{\beta}B_\delta^2}{r\lvert z\rvert}}\Big\}
\end{align*}
Now let $\lambda=\frac{T^{\beta}B_\delta^2}{r\lvert z\rvert}$. Therefore,
\begin{align*}
	e^{ay+b\tau}&=\exp\Big\{\frac{p\big[\lambda y-2y\big]+p^2\big[\lambda-\frac{1}{4}\lambda^2-1\big]}{2p\lambda}\Big\}\\
	&=\exp\Big\{\frac{y(\lambda -2)+\frac{1}{4}p(\lambda-2)^2}{2\lambda}\Big\}\\
	&=\exp\Big\{(\lambda -2)\frac{y+\frac{1}{4}p(\lambda-2)}{2\lambda}\Big\}.
\end{align*}
Thus,
\begin{align*}
	C(z,\tau)=e^{-r\tau}\frac{1}{\sqrt{2\tau\pi B_\delta^2T^{2\beta}}}e^{\frac{y(\lambda -2)+\frac{1}{4}p(\lambda-2)^2}{2\lambda}} \int_{-\infty}^\infty e^{-\frac{(y-\gamma)^2}{2p\lambda}}
	\max[\lvert z\rvert-\ln(K),0]d\gamma
\end{align*}
Since $\lvert z\rvert =\big\lvert \xi_{\delta,W_\delta}^\beta[\gamma]\big\rvert$, and $\xi_{t,W_t}^{-\beta}\big[\xi_{\delta,W_\delta}^{\beta}[\gamma]\big]=\gamma$ we have
\begin{align}
	C(z,\tau)=&e^{-r\tau}\frac{1}{\sqrt{4\pi p\lambda}}e^{\frac{y(\lambda -2)+\frac{1}{4}p(\lambda-2)^2}{2\lambda}}\int_{\xi_{t,W_t}^{-\beta}[\ln(K)]}^{\infty} e^{-\frac{(y-\gamma)^2}{2p\lambda}}
	(\xi_{\delta,W_\delta}^\beta[\gamma]-\ln(K))d\gamma\notag\\
	=&e^{-r\tau}e^{\frac{y(\lambda -2)+\frac{1}{4}p(\lambda-2)^2}{2\lambda}}\frac{1}{\sqrt{4\pi p\lambda}}\notag\\
	&\underbrace{\Big[\int_{\xi_{t,W_t}^{-\beta}[\ln(K)]}^{\infty} e^{-\frac{(y-\gamma)^2}{2p\lambda}}\xi_{\delta,W_\delta}^\beta[\gamma]d\gamma-\ln(K)\int_{\xi_{t,W_t}^{-\beta}[\ln(K)]}^{\infty}e^{-\frac{(y-\gamma)^2}{2p\lambda}}d\gamma\Big]}_{I_1-I_2}.\label{I12}
\end{align}
Let $M=\frac{y-\gamma}{\sqrt{2p\lambda}}$. We have $dM=\frac{-d\gamma}{\sqrt{2p\lambda}}$. Hence,
\begin{align}
	I_1=&\lvert z\rvert\int_{\frac{\ln(X)-\ln(\xi_{t,W_t}^{-\beta}[\ln(K)])}{\sqrt{2p\lambda}}}^{-\infty}  e^{-\frac{1}{2}M^2}(-\sqrt{2p\lambda})dM,\label{I1}\\
	I_2=&\ln(K)\int_{\frac{\ln(X)-\ln(\xi_{t,W_t}^{-\beta}[\ln(K)])}{\sqrt{2p\lambda}}}^{-\infty} e^{-\frac{1}{2}M^2}(-\sqrt{2p\lambda})dM.\label{I2}
\end{align}
Substituting \ref{I1} and \ref{I2} in \ref{I12} yields
\begin{align*}
	C(z,\tau)&=e^{-r\tau}e^{\frac{y(\lambda -2)+\frac{1}{4}p(\lambda-2)^2}{2\lambda}}\frac{1}{\sqrt{4\pi p\lambda}}\big[\lvert z\rvert-\ln(K)\big]\notag\\
	&\quad\int_{-\infty}^{\frac{\ln(X)-\ln(\xi_{t,W_t}^{-\beta}[\ln(K)])}{\sqrt{2p\lambda}}}  e^{-\frac{1}{2}M^2}(\sqrt{2p\lambda})dM\\
	&=e^{-r\tau}e^{\frac{y(\lambda -2)+\frac{1}{4}p(\lambda-2)^2}{2\lambda}}\big[\lvert z\rvert-\ln(K)\big]\notag\\
	&\quad\frac{1}{\sqrt{2\pi}}\int_{-\infty}^{\frac{\ln(X)-\ln(\xi_{t,W_t}^{-\beta}[\ln(K)])}{\sqrt{2p\lambda}}}  e^{-\frac{1}{2}M^2}dM
\end{align*}
Using lemma \ref{key} yields $\xi_{t,W_t}^{-\beta}[\ln(K)]=\ln(K)$. Therefore,
\begin{align*}
	C(z,\tau)&=e^{-r\tau}e^{\frac{y(\lambda -2)+\frac{1}{4}p(\lambda-2)^2}{2\lambda}}\big[\lvert z\rvert-\ln(K)\big]N\Big[\frac{\ln(X)-\ln[\ln(K)]}{\sqrt{2p\lambda}}\Big].
\end{align*}
Now taking $d=\frac{\ln(X)-\ln[\ln(K)]}{\sqrt{2p\lambda}}$ yields.
\begin{align*}
	C(z,\tau)&=e^{-r\tau}e^{\frac{y(\lambda -2)+\frac{1}{4}p(\lambda-2)^2}{2\lambda}}\big[\lvert z\rvert-\ln(K)\big]N[d],
\end{align*}
where
$$p=r\lvert z\rvert\tau T^{\beta},\quad \lambda=\frac{T^\beta B_\delta^2}{r\lvert z\rvert},\quad B_\delta=\frac{q}{\delta^{\beta-1}}+\frac{\sigma}{\delta^\beta}, \quad q=\mu-\frac{1}{2}\sigma^2.$$
\end{proof}

\section{Conclusion}\label{sec7}
The right time to leave a trading position is always unknown for market participants. In this paper, we presented a model using log-ergodic processes to estimate a time interval for exiting a trading position with profit and in the worst case with no loss.

Also, using the irrational rotation, we studied the European call option pricing problem on a unit circle. We substituted the time average of a mean ergodic process with the expectation in our calculations since the irrational rotation
is ergodic, providing a new framework for studying financial mathematics problems using ergodic theory.

Additionally, we solved the ergodic partial differential equation of Black-Scholes, introduced in our recent work \cite{116}. We have found a unique solution to this equation, which includes the inhibition degree parameter $\beta$.

Financial markets can be studied using the theory of dynamical systems, as they have the property of change of state concerning time. Studying the financial markets using ergodic theory has two main advantages: First, it provides a new approach to solving and modeling financial and economic problems. Second, using the time average instead of the expectation in some of the calculations makes it easier to study the markets in the long run. On the other hand, the existence of only one descriptive parameter, $\beta$, shows that our work is in the early stages, and there is much work to do in future studies for constructing efficient models.

\section*{Author Declarations}
\subsection*{Funding Information}
Not applicable
\subsection*{Declaration of Interest}
The authors have no conflicts of interest to declare. Both authors have seen and agree with the contents of the manuscript, and there is no financial interest to report.
\subsection*{Ethics Approval}
Not applicable
\subsection*{Consent to Participate}
Not applicable
\subsection*{Consent for Publication}
Not applicable
\subsection*{Data Availability Statement}
Not applicable
\subsection*{Code Availability Statement}
The MATLAB scripts and packages were used for simulation and process generations.
\subsection*{Author's Contributions}
Both authors, K. Firouzi and Dr. M. Jelodari Mamaghani, contributed equally to this work. K. Firouzi wrote the manuscript under the supervision of Dr. M. Jelodari Mamaghani. Both authors read and approved the final manuscript.

\bibliographystyle{plain}
\bibliography{refr}


\end{document}